\documentclass[11pt]{amsart}

\usepackage{graphicx}
\usepackage[small]{caption}
\usepackage{datetime,verbatim,array}
\usepackage{amsthm,amsmath,amsfonts,amssymb}
\usepackage{mathrsfs,rotating}
\usepackage[usenames]{color}
\usepackage[colorlinks,citecolor=black,filecolor=black,linkcolor=black,urlcolor=black]{hyperref}
\usepackage{mathptmx}	%roman=Times, math=Times where possible
\usepackage[all,cmtip]{xy}

\usepackage{pinlabel}

\newtheorem{theorem}{Theorem}[section]

\newtheorem{lemma}[theorem]{Lemma}
\newtheorem{proposition}[theorem]{Proposition}
\newtheorem{corollary}[theorem]{Corollary}

\def\N{\mathbb{N}}				%Natural numbers
\def\R{\mathbb{R}}					%Real numbers
\def\C{\mathcal{C}}				
\def\A{\mathcal{A}}				
\def\pminus{+-}
\def\mplus{-+}
\def\mminus{--}

\newcommand{\nc}{\newcommand}

\nc{\imin}{i_{\textrm{min}}}

\nc{\p}[1]{\medskip\noindent{\em #1.}}
\nc{\margin}[1]{\marginpar{\scriptsize #1}}

\title[Efficient geodesics and an effective algorithm for distance]{Efficient geodesics and an effective algorithm for distance in the complex of curves}

\author{Joan Birman, Dan Margalit, and William Menasco}

\address{Joan Birman\\
Department of Mathematics, Barnard--Columbia\\
2990 Broadway\\
New York, NY 10027, USA\\ jb@math.columbia.edu}

\address{Dan Margalit \\ School of Mathematics\\ Georgia Institute of Technology \\ 686 Cherry St. \\ Atlanta, GA 30332 \\  margalit@math.gatech.edu}

\address{William W. Menasco\\
Department of Mathematics\\
University at Buffalo--SUNY\\
Buffalo, NY 14260-2900, USA\\ menasco@buffalo.edu}

\thanks{The first author gratefully acknowledges partial support from the Simons Foundation, under Collaborative Research Award \#245711.
The second author gratefully acknowledges support from the National Science Foundation.}

\begin{document}

\vspace*{-1in}

\maketitle

\vspace*{-.25in}

\begin{abstract}  
We give an algorithm for determining the distance between two vertices of the complex of curves.  While there already exist such algorithms, for example by Leasure, Shackleton, and Webb, our approach is new, simple, and more effective for all distances accessible by computer.  Our method gives a new preferred finite set of geodesics between any two vertices of the complex, called efficient geodesics, which are different from the tight geodesics introduced by Masur and Minsky.
\end{abstract}

\begin{figure}[h!]
\centerline{\includegraphics[width=1\textwidth]{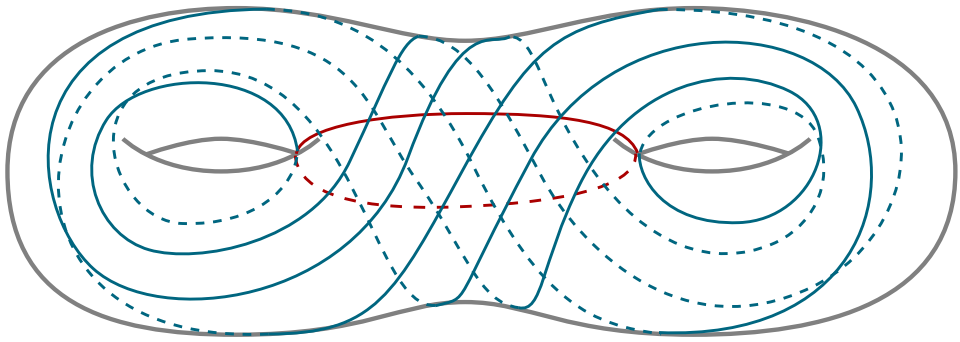}}
\caption{Vertices of $\C(S_2)$ with distance 4 and intersection number 12; this is the smallest possible intersection for vertices with distance 4}
\label{F:3dsym}
\end{figure}

\vspace*{-3ex}

\section{Introduction}
\label{sec:intro}

The complex of curves $\C(S)$ for a compact surface $S$ is the simplicial complex whose vertices correspond to isotopy classes of essential simple closed curves in $S$ and whose edges connect vertices with disjoint representatives.  We can endow the 0-skeleton of $\C(S)$ with a metric by defining the distance between two vertices to be the minimal number of edges in any edge path between the two vertices.

The geometry of $\C(S)$---especially the large-scale geometry---has been a topic of intense study over the past two decades, as there are deep applications to the theories of 3-manifolds, mapping class groups, and Teichm\"uller space; see, e.g., \cite{Minsky}. The seminal result, due to Masur and Minsky in 1996, states that $\C(S)$ is $\delta$-hyperbolic \cite{MM1}.  Recently, several simple proofs of this fact have been found, and it has been shown that $\delta$ can be chosen independently of $S$; see  \cite{A,BB,CRS,HPW,SP}.  

In 2002, Leasure \cite[\S 3.2]{L} found an algorithm to compute the distance between two vertices of $\C(S)$, and since then other algorithms have been devised by Shackleton \cite{S}, Webb \cite{Webb}, and Watanabe \cite{Watanabe1}.  About his algorithm, Leasure says:
\begin{quote}
\emph{We do not mention this in the belief that anyone will ever implement
it. The novelty is that finding the exact distance between two curves in the
curve complex should be so awkward.}
\end{quote}
One goal of this paper is to give an algorithm for distance---the efficient geodesic algorithm---that actually can be implemented, at least for small distances.  The third author and Glenn, Morrell, and Morse \cite{GMMM} have in fact already developed an implementation of our algorithm, called Metric in the Curve Complex \cite{MICC}.  Their program is assembling a data bank of examples as we write.

\medskip

\p{Known examples} Let $S_g$ denote a closed, connected, orientable surface of genus $g$ and let $\imin(g,d)$ denote the minimal intersection number for vertices of $\C(S_g)$ with distance $d$.   The Metric in the Curve Complex program has been used to show that:
\begin{enumerate}
\item $\imin(2,4)=12$ and 
\item $\imin(3,4) \leq 21$.
\end{enumerate}
The highly symmetric example in Figure~\ref{F:3dsym}---which realizes $\imin(2,4)$---was discovered using the program.  See Section~\ref{S:computation} for a discussion of this example and a proof using the methods of this paper that the distance is actually 4.

We are only aware of one other explicit picture in the literature of a pair of vertices of $\C(S_2)$ that have distance four, namely, the example of Hempel that appears in the notes of Saul Schleimer \cite[Figure 2]{SS} (see \cite[Example 1.6]{GMMM} for a proof that the distance is 4).  This example has geometric intersection number 25.   

Using the bounded geodesic image theorem \cite[Theorem 3.1]{MM2} of Masur and Minsky (as quantified by Webb \cite{Webb3}) it is possible to explicitly construct examples of vertices with any given distance; see \cite[Section 6]{S}.  We do not know how to keep the intersection numbers close to the minimum with this method, but Aougab and Taylor did in fact use this method to give examples of vertices of arbitrary distance whose intersection numbers are close to the minimum in an asymptotic sense; see their paper \cite{AT} for the precise statement.

\medskip

\p{Local infinitude} One reason why computations with the complex of curves are so difficult is that it is locally infinite and moreover there are infinitely many geodesics (i.e. shortest paths) between most pairs of vertices.  Masur and Minsky  \cite{MM2} addressed this issue by finding a preferred set of geodesics, called tight geodesics, and proving that between any two vertices there are finitely many tight geodesics; see Section~\ref{sec:vs} for the definition.
Our first goal is to give a new class of geodesics that still has finitely many elements connecting any two vertices but is more amenable to certain computations.

\smallskip

\p{Efficient geodesics} Our approach to geodesics in $\C(S)$ is defined in terms of intersections with arcs.  First, suppose that $\gamma$ is an arc in $S$ and $\alpha$ is a simple closed curve in $S$.  We say that $\gamma$ and $\alpha$ are in \emph{minimal position} if $\alpha$ is disjoint from the endpoints of $\gamma$ and the number of points of intersection of $\alpha$ with $\gamma$ is smallest over all simple closed curves that are homotopic to $\alpha$ through homotopies that do not pass through the endpoints of $\gamma$.  

Let $v_0, \dots, v_n$ be a geodesic of length at least three in $\C(S)$, and let $\alpha_0$, $\alpha_1$, and $\alpha_n$ be representatives of $v_0$, $v_1$, and $v_n$ that are pairwise in minimal position (this configuration is unique up to isotopy of $S$).  A \emph{reference arc} for the triple $\alpha_0, \alpha_1, \alpha_n$ is an arc $\gamma$ that is in minimal position with $\alpha_1$ and whose interior is disjoint from $\alpha_0 \cup \alpha_n$; such arcs were considered by Leasure \cite[Definition 3.2.1]{L}.

We say that the oriented geodesic $v_0,\dots,v_n$ is \emph{initially efficient} if 
\[ |\alpha_1\cap\gamma| \leq n-1 \]
for all choices of reference arcs $\gamma$ (this is independent of the choices of $\alpha_0$, $\alpha_1$, and $\alpha_n$ by the uniqueness statement above).  
Finally, we say that $v=v_0, \dots, v_n=w$ is \emph{efficient} if the oriented geodesic $v_k,\dots,v_n$ is initially efficient for each $0 \leq k \leq n-3$ and the oriented geodesic $v_n,v_{n-1},v_{n-2},v_{n-3}$ is also initially efficient.

We emphasize that to test the initial efficiency of $v_k,\dots,v_n$ we should look at reference arcs for the triple $v_k$, $v_{k+1}$, and $v_n$ and we allow $n-k-1$ points of intersection of (a representative of) $v_{k+1}$ with any such reference arc.

\p{Existence of efficient geodesics} 
Our main result is that efficient geodesics always exist, and that there are finitely many between any two vertices.

\begin{theorem}
\label{T:simplify}
Let $g \geq 2$.  If $v$ and $w$ are vertices of $\C(S_g)$ with $d(v,w) \geq 3$, then there exists an efficient geodesic from $v$ to $w$.  What is more, there is an explicitly computable list of at most 
\[ n^{6g-6} \]
vertices $v_1$ that can appear as the first vertex on an initially efficient geodesic \[ v=v_0,v_1,\dots,v_n=w. \]
In particular, there are finitely many efficient geodesics from $v$ to $w$.
\end{theorem}

We emphasize that our theorem is only for closed surfaces; see the discussion on page~\pageref{fail} about surfaces with boundary for an explanation.  We also mention that this theorem is stronger than Theorem 1.1 in the first version of this paper \cite{v1}; see Proposition~\ref{P:improvement} and the accompanying discussion.

\p{Finitely many reference arcs} While \emph{a priori} there are infinitely many reference arcs that need to be checked in the definition of initial efficiency there are in fact finitely many.  Indeed, let $\alpha_0$, $\alpha_1$, and $\alpha_n$ be representatives of $v$, $v_1$, and $w$ that have minimal intersection pairwise. Since $d(v,w) \geq 3$ it follows that $\alpha_0$ and $\alpha_n$ \emph{fill} $S$, which means that they together decompose $S$ into a collection of polygons.  We can endow each such polygon with a Euclidean metric and replace each segment of $\alpha_1$ in each polygon with a straight line segment.  

There are finitely many non-rectangular polygons in the decomposition since each $2k$-gon contributes $-(k-2)/2$ to $\chi(S)$.  And each reference arc in a rectangular region is parallel to one in a non-rectangular region.    Thus in order to check initial efficiency, it is enough to consider reference arcs that lie in a non-rectangular polygonal region.  Furthermore, it is enough to consider reference arcs that are straight line segments connecting the midpoints of the $\alpha_0$-edges of a polygon.  Indeed, such an arc is necessarily in minimal position with $\alpha_1$ and any other reference arc can be extended to such a reference arc.  

In the special case that the reference arc connects the midpoints of $\alpha_0$-edges that are consecutive in a polygon, the reference arc is parallel to the $\alpha_n$-edge in between.  In this case points of $\alpha_1 \cap \gamma$ are in bijection with points of $\alpha_1 \cap \alpha_n$, and so the definition of initial efficiency can be translated into a statement about intersections of $\alpha_1$ with $\alpha_n$; see Proposition~\ref{P:improvement} below.

\p{Finitude of efficient geodesics} The main point of Theorem~\ref{T:simplify} is the existence statement; the finiteness statement can be dispensed with immediately.  Indeed, for any geodesic $v_0,\dots,v_n$ let $\alpha_0$, $\alpha_1$, and $\alpha_n$ be representatives of $v_0$, $v_1$, and $v_n$ that have minimal intersection pairwise.  As above,  $\alpha_0$ and $\alpha_n$ decompose $S_g$ into a collection of polygons.

If we cut $S_g$ along $\alpha_0$ we obtain a surface $S_g'$ with two boundary components on which $\alpha_n$ becomes a collection of arcs.  The $\alpha_n$-arcs cut $S_g'$ into a collection of even-sided polygons.  We can choose reference arcs in $S_g'$ that are disjoint from each other, that have interiors disjoint from the $\alpha_n$-arcs, and that cut $S_g'$ into hexagons.  Such a collection is obtained by taking one reference arc parallel to each parallel family of arcs of $\alpha_n$ and then taking additional reference arcs cutting across any remaining polygons with more than six sides.

An Euler characteristic count shows that any such collection of reference arcs has $6g-6$ elements.  Also, since $\alpha_1$ is disjoint from $\alpha_0$ the curve $\alpha_1$ is determined up to homotopy by the number of intersections it has with each reference arc.  By the definition of initial efficiency, each of these intersection numbers is between 0 and $n-1$.  This gives the bound stated in Theorem~\ref{T:simplify}.

\p{Discussion of the proof} Our method for proving Theorem~\ref{T:simplify} is detailed in Section~\ref{sec:simplify}.  Briefly, the idea is to show that if some geodesic $v=v_0, \dots, v_n=w$ is not initially efficient then we can modify $v_1,\dots,v_{n-1}$ by surgery in order to reduce the intersection of $v_1$ with $v_0$ and $v_n$.  The basic surgeries we use in our proof are not new.  The crucial point---and our new idea---is that it is usually not possible to reduce intersection by modifying a single vertex; rather, it is often the case that we can reduce intersection by modifying a sequence of vertices all at the same time. 

 \begin{figure}[htbp!]
\labellist
\small\hair 2pt
 \pinlabel {$v_n$} [ ] at -9 50
 \pinlabel {$v_0$} [ ] at 9 7
 \pinlabel {$v_1$} [ ] at 22 74
 \pinlabel {$v_1'$} [ ] at 208 50
 \pinlabel {$v_1$} [ ] at 327 74
 \pinlabel {$v_2$} [ ] at 385 25
 \pinlabel {$v_1'$} [ ] at 475 51
 \pinlabel {$v_2'$} [ ] at 524 60
\endlabellist
\centering{\includegraphics[width=.95\textwidth]{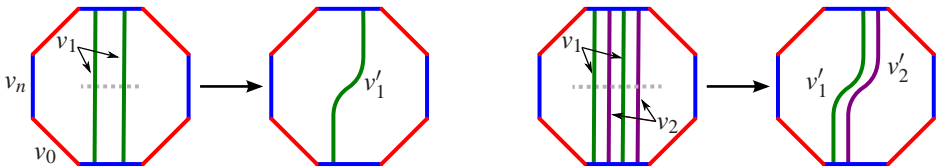}}
\caption{\emph{Left:} two arcs of $v_1$ and a simplifying surgery; \emph{Right:} arcs of $v_1$ and $v_2$ and a simplifying surgery}
\label{F:idea}
\end{figure}

Here is what we mean by this.  Suppose we have a geodesic $v_0,...,v_n$. Say there is a $v_0$-$v_n$ polygon with two parallel arcs of $v_1$ as in the first picture in the left-hand side of Figure~\ref{F:idea}.  Then we can perform a surgery along the dotted reference arc as in the figure in order to find a vertex $v_1'$ that is simpler in that it has fewer intersections with $v_0$ and $v_n$.  The vertex $v_1'$ can replace $v_1$ in the geodesic since the surgery did not create any intersections with $v_0$ or $v_2$.

Now suppose we have four parallel arcs of $v_1$, $v_2$, $v_1$, and $v_2$ (in order) as in the right-hand side of Figure~\ref{F:idea}.  We cannot surger $v_1$ as in the previous paragraph because this would create an intersection with $v_2$---an arc of $v_2$ is in the way.  However, we can perform surgery simultaneously on $v_1$ and $v_2$ along the dotted arc as in the figure.  This gives two new vertices $v_1'$ and $v_2'$ and again we can replace $v_1$ and $v_2$ with these new, simpler vertices.

Our basic strategy is to show that whenever we have an inefficient geodesic we can find a similar surgery in order to reduce intersection with $v_0$ and $v_n$.  If the reference arc only sees $v_1$ and $v_2$ then the surgeries in the previous two paragraphs apply.  The problem is that when there are more vertices $v_i$ involved, there are more and more complicated surgeries needed, and the combinatorics get to be unwieldy; look ahead to Figures~\ref{F:repetition} and~\ref{F:glock surgery} for examples of more complicated surgeries.

To deal with this problem, we introduce a new tool, the dot graph.  This is a graphical representation of the sequence of vertices $v_i$ seen along a reference arc; there is a dot at the point $(k,i)$ in the plane if the $k$th vertex along the arc is $v_i$ (see Figure~\ref{F:dot graph} below).  The existence of a simplifying surgery is translated into the existence of certain two-dimensional shapes in the dot graph (see Figure~\ref{F:types} below).  In this way, the unwieldy combinatorial problem becomes a manageable geometric one.

\p{Efficiency versus tightness} We already mentioned that there are finitely many tight geodesics between two vertices of $\C(S_g)$ and so Theorem~\ref{T:simplify} gives a second finite class of geodesics connecting two vertices of $\C(S_g)$.  The next proposition shows that the class of efficient geodesics is genuinely new.

\begin{proposition}
\label{prop:vs}
Let $g \geq 2$. In $\C(S_g)$ there are geodesics of length three that are...
\begin{enumerate}
\item efficient and tight, 
\item tight but not efficient, and
\item efficient but not tight.
\end{enumerate}
\end{proposition}
We do not know if between any two vertices there always exists a geodesic that is efficient and tight.

Proposition~\ref{prop:vs} is proved by explicit construction;  see Section~\ref{sec:vs}.  The most subtle point is the third one, as it is in general not easy to prove that a given geodesic is not contained in any tight multigeodesic.

While the examples of geodesics in Proposition~\ref{prop:vs} all have length three, we expect that the result holds for all distances at least three.  It is also worth noting that our constructions are all delicate: it is not obvious how to modify our examples in order to obtain infinite families of examples.

\p{The efficient geodesic algorithm} We now explain how  Theorem~\ref{T:simplify} can be used in order to give an algorithm for distance in $\C(S_g)$, which we call the \emph{efficient geodesic algorithm}.  It is straightforward to determine if the distance between two vertices is 0, 1, or 2.  So assume that for some $k \geq 2$ we have an algorithm for determining if two vertices of $\C(S_g)$ have distance $0, \dots, k$.  We would like to give an algorithm for determining if the distance between two vertices is $k+1$.  

To this end, let $v$ and $w$ be two vertices of $\C(S_g)$.  By induction we can check if $d(v,w) \leq k$.  If not, then as in Theorem~\ref{T:simplify} we can explicitly list all possible vertices $v_1$ on an efficient geodesic $v=v_0,\dots,v_{k+1}=w$.  If $d(v_1,w) = k$ for some choice of $v_1$, then $d(v,w) = k+1$; otherwise it follows from Theorem~\ref{T:simplify} (the existence of efficient geodesics) that $d(v,w) \neq k+1$.

\begin{corollary}
\label{cor:algorithm}
The efficient geodesic algorithm computes distance in $\C(S_g)$.
\end{corollary}

The special case of the efficient geodesic algorithm when the distance is four was explained to us by John Hempel and served as inspiration for the cases of larger distance.

\medskip

\p{Comparison with previously known algorithms} Our efficient geodesic algorithm is in the same spirit as the algorithms of Leasure, Shackleton, and Watanabe for computing distance in $\C(S_g)$.  All three show that there is a function $F$ of three variables so that for any two vertices $v$ and $w$ of $\C(S_g)$ with $d(v,w)=n$ there is a geodesic $v=v_0,\dots,v_n=w$ with $i(v_1,w)$ bounded above by $F(g,n,i(v,w))$.  This gives an algorithm in the same way as our efficient geodesic algorithm, since there is an explictly computable list of $v_1$ with $i(v,v_1)=0$ and $i(v_1,w) \leq F(g,n,i(v,w))$.  While the theorems of Leasure, Shackleton, and Watanabe apply to surfaces that are not closed, we restrict here to the case of closed surfaces for simplicity.

Our approach also gives such a function $F$.  By only considering reference arcs that are parallel to arcs of $\alpha_n \setminus \alpha_0$ (where $\alpha_0$ and $\alpha_n$ are minimally-intersecting representatives of $v_0$ and $v_n$), we deduce that for any initially efficient geodesic $v=v_0,\dots,v_n=w$ we have $i(v_1,v_n) \leq (n-2)i(v,w)$ (this uses a slight strengthening of a special case of Theorem~\ref{T:simplify}; see Proposition~\ref{P:improvement} below).  So we can take \[F_{BMM}(g,n,i(v,w)) = (n-2)i(v,w).\]
However, this bound does not use the full strength of initial efficiency as it does not give information as to how these points of intersection are distributed along $\alpha_n$ nor does it take into account reference arcs  that are not parallel to $\alpha_n$.

Leasure's function is 
\[ F_L(g,n,i(v,w)) = (6(6g-2)+2)^{n} i(v,w). \]
We can illustrate the improvement of our methods over Leasure's with the example in $\C(S_2)$ from Figure~\ref{F:3dsym}.  To prove the distance is 4, we can suppose for contradiction that it is 3.  According to Leasure, if $v_1$ is the first vertex we meet on a length 3 geodesic from $v$ to $w$, then we can choose $v_1$ so that it satisfies 
\[ i(v_1,w) \leq (6(6g-2)+2)^3 i(v,w) = 62^3 \cdot 12 = 2,859,936. \]
By contrast, any $v_1$ on an efficient geodesic of length 3 satisfies $i(v_1,w) \leq 12$ and, what is more, we know there is at most one intersection of $v_1$ along each edge of the polygonal decomposition of $S_2$ determined by $v$ and $w$ (cf. Proposition~\ref{P:improvement} below). Because of these strong restrictions, the computation can be carried out by hand, and in fact we apply the algorithm by hand to this example in Section~\ref{S:computation}.

Shackleton's function depends only on $i(v,w)$ and $g$, but not $d(v,w)$.  As explained by Watanabe \cite{Watanabe1}, Shackleton's function is 
\[ F_S(g,n,i(v,w)) = i(v,w) (4^5(3g-3)^3)^{2 \log_2 i(v,w)}. \]

Watanabe recently improved on Shackleton's result by replacing the exponential function with a linear one.  His work, like Webb's, uses the theory of tight geodesics.  Specifically, Watanabe's function is:
\[ F_{W}(g,n,i(v,w)) = R_g \, i(v,w) \]
where $R_g=(3g-3) \cdot 2^{(3M+1)^3(2g-2)^{3g-3}}$ and $M$ is the minimal possible constant in the bounded geodesic image theorem.  Since $R_g$ is independent of $n$, it follows that when $n$ is large compared to $g$ Watanabe's bounds give a better algorithm for distance than the efficient geodesic algorithm.  However, the smallest known upper bound for $M$ is 102 (see \cite{WebbThesis}), and so even for $g=2$, we have
\[ R_g = 3 \cdot 2^{231,475,544} > 10^{69,681,082}.\]
Thus, even for $g=2$ and some unimaginably large distances, our algorithm is more effective.

In the appendix we will explain Webb's algorithm for computing distance via tight geodesics.  As explained to us by Webb \cite{Webb2}, his methods give a corresponding function that again only depends on $g$:
\[ F_{W'}(g) = 
\frac{(6g-6) \left((4g-5)^{21} - 4g+5\right)}{2g-3}, \]
which for $g=2$ equals 62,762,119,200.

A more appropriate comparison with Webb's algorithm is to compare the number of vertices $v_1$ that need to be tested instead of the quantity $i(v_1,v_n)$.  In Webb's algorithm, this number is bounded above by:
\[ 2^{(72g+12)\min\{n-2,21\}} (2^{6g-6}-1  ) \]
(here we are really counting the number of candidate simplices $\sigma_1$ along a multigeodesic from $v$ to $w$); see the appendix of this paper for an explanation.  
On the other hand, our Theorem~\ref{T:simplify} states that the number of candidate vertices $v_1$ along an efficient geodesic $v_0,\dots,v_n$ is bounded above by $n^{6g-6}$.  Our bound is smaller than Webb's when $\min\{n-2,21\}=n-2$.  In the case that $\min\{n-2,21\}=21$ we estimate Webb's bound from below by $2^{(72g)(21)}$ and we find that our bound is smaller than Webb's for all distances less than $2^{21(12)}$, which is approximately $10^{75}$.  We conclude that among all known algorithms for distance in $\C(S)$ our methods are by far the most effective for all distances accessible by modern computers.

\p{Acknowledgments}  We would like to thank Ken Bromberg, Chris Leininger, Yair Minsky, Kasra Rafi, and Yoshuke Watanabe for helpful conversations. We are especially grateful to John Hempel for sharing with us his algorithm, to Richard Webb for sharing many ideas and details of his work, and to Tarik Aougab for many insightful comments, especially on the problem of constructing geodesics that are not tight.  Finally, we would like to thank Paul Glenn, Kayla Morrell, and Matthew Morse for supplying numerous examples generated by their program Metric in the Curve Complex.

%%%
%%%
%%%

\section{Examples}
\label{S:computation}

In this section we do two things.  First we illustrate the efficient geodesic algorithm by applying it to the  example from Figure~\ref{F:3dsym}.  Then we prove Proposition~\ref{prop:vs} by giving explicit examples for each of the three statements.  All of the examples will be presented in terms of the branched double cover of $S_g$ over the sphere, which we now explain.

\p{The branched double cover} Let $X_{2g+2}$ denote a sphere with $2g+2$ marked points.  The double cover branched over the marked points is the closed surface $S_g$.  The preimage of a simple arc in $X_{2g+2}$ connecting two marked points is a nonseparating simple closed curve in $S_g$, and the preimage of a simple closed curve that surrounds $2k+1$ marked points is a separating simple closed curve in $S_g$ that cuts off a subsurface of genus $k$.  

Minimally intersecting curves and arcs in $X_{2g+2}$ lift to minimally intersecting curves and arcs in $S_g$.  This follows from the work of the first author and Hilden on the symmetric mapping class group \cite{bh}; see also the paper by Winarski \cite{Winarski}.  Also, if two minimally intersecting curves or arcs fill $X_{2g+2}$---meaning that the complementary components are all disks with at most one marked point each---then the preimages fill $S_g$ since the preimage of a disk with at most one marked point is a disk.

\subsection{An example of the efficient geodesic algorithm}\label{subsec:example} Consider the two arcs $\delta$ and $\epsilon$ in $X_6$ shown in the left-hand side of Figure~\ref{F:examplesphere} (we depict $X_{2g+2}$ by drawing $2g+2$ dots in the plane; by adding an unmarked point at infinity, we obtain the sphere with $2g+2$ marked points).  Let $v$ and $w$ denote the corresponding vertices of $S_2$, the two-fold branched cover over $X_6$.  We would like to show that the distance between $v$ and $w$ in $\C(S_2)$ is 4 (it so happens that $v$ and $w$ are the same as the vertices of $\C(S_2)$ shown in Figure~\ref{F:3dsym}, but we will not need this).  The distance between $v$ and $w$ in $\C(S_2)$ can be computed with the computer program Metric in the Curve Complex, but here we explain how to apply our algorithm by hand.

\begin{figure}[htbp!]
\labellist
\small\hair 2pt
 \pinlabel {$\delta$} [ ] at 22 77
 \pinlabel {$\epsilon$} [ ] at 133 63
 \pinlabel {$\delta$} [ ] at 291 44
 \pinlabel {$\epsilon^+$} [ ] at 250 83
 \pinlabel {$\epsilon^-$} [ ] at 325 10
\endlabellist
\centering{\includegraphics[width=.95\textwidth]{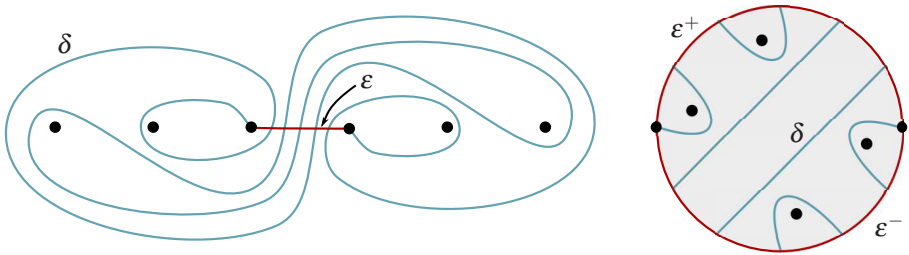}}
\caption{\emph{Left:} the arcs $\delta$ and $\epsilon$ in $X_6$ corresponding to the curves shown in Figure~\ref{F:3dsym}; \emph{Right:} the disk $\Delta$ obtained by cutting along $\epsilon$}
\label{F:examplesphere}
\end{figure}

First, we will show that $d(v,w) \leq 4$.  To do this, we observe that the horizontal line segment connecting the second and third marked points in the left-hand side of Figure~\ref{F:examplesphere} corresponds to a vertex $u$ in $\C(S_2)$ with $i(u,v)=i(u,w)=1$.  It follows that $d(u,v)=d(u,w)=2$ and by the triangle inequality that $d(v,w) \leq 4$.

If we cut $X_6$ along $\epsilon$, we obtain a disk $\Delta$, the shaded disk in the right-hand side of Figure~\ref{F:examplesphere}.  The boundary of $\Delta$ consists of two copies of $\epsilon$, say, $\epsilon^+$ and $\epsilon^-$, and in the figure points of $\epsilon^+$ and $\epsilon^-$ are identified in $X_6$ exactly when they lie on the same vertical line.   The arc $\delta$ becomes a collection of arcs in $\Delta$ as shown in the figure.  Since the arcs of $\delta$ cut $\Delta$ into a disjoint union of disks with at most one marked point each, it follows that $\delta$ and $\epsilon$ fill $S_2$ and so $d(v,w) \geq 3$.

It remains to use the efficient geodesic algorithm to show that $d(v,w) \geq 4$.  Assume that $d(v,w)$ were equal to three.  By Theorem~\ref{T:simplify} there is a path $v,v_1,v_2,w$ so that the number of intersections of $v_1$ with each arc of $w \setminus v$ is at most two (consider a reference arc parallel to the arc of $w \setminus v$).  Proposition~\ref{P:improvement} below gives an improvement: there is a choice of $v_1$ so that the intersection with each arc of $w \setminus v$ is at most one point.  Also, since $v,v_1,v_2,w$ is a path, this choice of $v_1$ satisfies $d(v_1,w) \leq 2$; in other words, (representatives of) $v_1$ and $w$ do not fill $S_2$.

A special feature of the genus two case is that every vertex of $\C(S_2)$ is obtained as the preimage of a curve or arc in $X_6$ (this again follows from the work of the first author with Hilden).  In this way, any $v_1$ as in the previous paragraph corresponds to an arc or curve $\beta$ in $\Delta$ that intersects each arc of $\delta$ in at most one point.  There are only six such candidates for $\beta$, namely the six straight line segments connecting marked points in the interior of $\Delta$.  It is straightforward to check that the arc in $X_6$ corresponding to each fills with $\delta$.  Therefore there is no $v_1$ as in the previous paragraph and we have $d(v,w) = 4$.

\subsection{Efficiency versus tightness}
\label{sec:vs}

We will now prove Proposition~\ref{prop:vs}---that there are geodesics in $\C(S_g)$ that are efficient and tight, geodesics that are efficient but not tight, and geodesics that are tight but not efficient.  First we recall the definition of a tight geodesic.

\p{Tight geodesics} A \emph{tight multigeodesic} is a sequence of simplices $\sigma_0,\dots,\sigma_n$ in $\C(S)$ where
\begin{enumerate}
\item $\sigma_0$ and $\sigma_n$ are vertices,
\item the distance between $v_i$ and $v_j$ is $|j-i|$ whenever $i \neq j$ and $v_i$ and $v_j$ are vertices of $\sigma_i$ and $\sigma_j$, respectively, and
\item for each $1 \leq i \leq n-1$ the simplex $\sigma_i$ can be represented as the union of the essential components of the boundary of a regular neighborhood in $S$ of minimally-intersecting representatives of $\sigma_{i-1}$ and $\sigma_{i+1}$.
\end{enumerate}
This definition is due to Masur and Minsky.\footnote{Masur and Minsky used the term ``tight geodesic,'' instead of ``tight multigeodesic,'' language we prefer to avoid because the object in question is not a geodesic.}  We will refer to any sequence of vertices $v_0,\dots,v_n$ with $v_i \in \sigma_i$ as a \emph{tight geodesic}.

\begin{proof}[Proof of Proposition~\ref{prop:vs}]

We begin with the first statement, there there are geodesics in $\C(S_g)$ that are both efficient and tight.  Consider the arcs $\delta_0$, $\delta_1$, $\delta_2$, and $\delta_3$ in $X_{12}$ shown in Figure~\ref{F:t and e}.  As above, each arc $\delta_i$ represents a vertex $v_i$ of $\C(S_5)$.  We have $d(v_0,v_3) = 3$ since $\delta_0$ and $\delta_3$ fill $X_{12}$.

To see that the  geodesic $v_0,v_1,v_2,v_3$ is efficient we first note that $\delta_1$ intersects only two regions of $X_{16}$ determined by $\delta_0$ and $\delta_3$.  One of these regions is a bigon with one marked point; the preimage of this is a rectangular region in $S_5$ and so it can be ignored.  The other region is a disk with no marked points and its preimage is a pair of disks in $S_5$.  The preimage of $\delta_1$ passes through each of these disks in $S_5$ once, whence the initial efficiency of $v_0,v_1,v_2,v_3$.  There is an obvious symmetry of $X_{12}$ reversing the geodesic and so $v_3,v_2,v_1,v_0$ is initially efficient.  Hence $v_0,v_1,v_2,v_3$ is indeed efficient.

\begin{figure}[h!]
 \labellist
\small\hair 2pt
\pinlabel $\delta_0$ at -2 56
\pinlabel $\delta_2'$ at 100 42
\pinlabel $\delta_2''$ at 175 42
\pinlabel $\delta_2$ at 26 42
\pinlabel $\delta_1'$ at 100 19
\pinlabel $\delta_1''$ at 175 19
\pinlabel $\delta_1$ at 26 19
\pinlabel $\delta_3$ at -2 7
\endlabellist
 	\centerline{\includegraphics[scale=1.25]{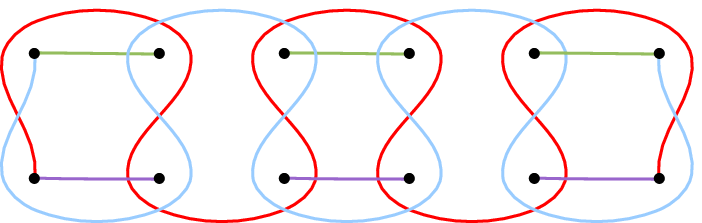}}
	\caption{Arcs giving a geodesic in $\C(S_5)$ that is both efficient and tight} 
	\label{F:t and e}
\end{figure}

Let $v_i'$ and $v_i''$ denote the vertices of $\C(S_5)$ corresponding to the arcs $\delta_i'$ and $\delta_i''$.  The simplices $\sigma_1=\{v_1,v_1',v_1''\}$ and $\sigma_2=\{v_2,v_2',v_2''\}$ give a tight multigeodesic $v_0,\sigma_1,\sigma_2,v_3$ with $v_1 \in \sigma_1$ and $v_2 \in \sigma_2$, certifying that $v_0,v_1,v_2,v_3$ is a tight geodesic.  (To verify this, note that the preimage in $S_g$ of a disk with two marked points in $X_{2g+2}$ is an annulus.)

For any odd $g > 2$ a straightforward generalization applies.  For even $g$ a slight modification is needed; for instance to obtain an analogous example for $S_4$ from Figure~\ref{F:t and e}, we move the left-hand endpoints of $\delta_0$ and $\delta_3$ together and we move the right-hand endpoints together as well, giving a collection of arcs in $X_{10}$.

\begin{figure}[h!]
 \labellist
\small\hair 2pt
 \pinlabel {$\delta_0$} [ ] at 10 35
 \pinlabel {$\delta_2$} [ ] at 129 42
 \pinlabel {$\delta_3$} [ ] at 69 63
\endlabellist
 	\centerline{\includegraphics[scale=1.1]{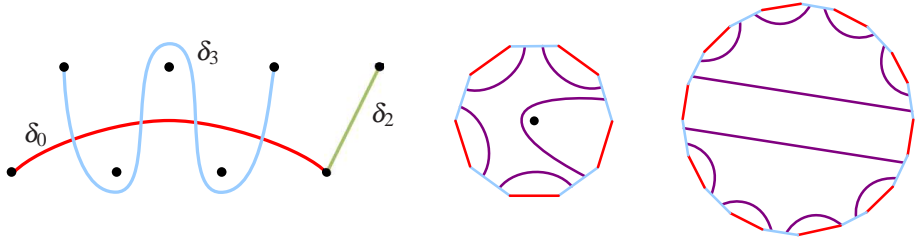}}
	\caption{\emph{Left:} Arcs in $X_8$ giving a tight geodesic in $\C(S_3)$; \emph{Middle:} The outer 10-gon in $X_8$ cut along $\delta_0$ and $\delta_3$ shown with arcs of $\delta_1$; \emph{Right:} The preimage of the 10-gon in $S_3$ shown with the preimage of $\delta_1$}
	\label{F:t not e}
\end{figure}

 We now give examples of geodesics that are tight but not efficient.  Consider the arcs $\delta_0$, $\delta_2$, and $\delta_3$ in  shown in the left-hand side of Figure~\ref{F:t not e}.  Let $\delta_1$ be the boundary of a regular neighborhood of $\delta_0 \cup \delta_2$; this $\delta_1$ is a curve surrounding three marked points.  Let $v_0,v_1,v_2,v_3$ be the corresponding path in $\C(S_3)$.  We have $d(v_0,v_3) \geq 3$ since $\delta_0$ and $\delta_3$ fill $X_{8}$.  By definition $v_0,v_1,v_2,v_3$ is tight at $v_1$ (meaning that the third part of the definition of a tight multigeodesic is satisfied for $i=1$) and it is straightforward to check that it is tight at $v_2$; so more than being contained in a tight multigeodesic, the given geodesic is itself a tight multigeodesic (in other words, $v_0,v_1,v_2,v_3$ is a tight multigeodesic with a single associated tight geodesic).

We will now show that the oriented geodesic $v_0,v_1,v_2,v_3$ is not initially efficient.  If we cut $X_8$ along $\delta_0$ and $\delta_3$ there is a single region that is not a bigon with one marked point, namely, the region containing the (umarked!) point at infinity.  There are five arcs of $\delta_1$ in this disk as shown in the middle picture of Figure~\ref{F:t not e} (the exact configuration relative to the marked point is important here).  The preimage of this 10-gon in $S_3$ is a 20-gon, and the arcs of the preimage of $\delta_1$ are arranged as in the right-hand side of Figure~\ref{F:t not e}.  It is easy to find a reference arc in this polygon that intersects the preimage of $\delta_1$ in more than two points.  Thus $v_0,v_1,v_2,v_3$ is not initially efficient; of course this implies that $v_0,v_1,v_2,v_3$ is not efficient.  The generalization to higher genus should be clear.

\begin{figure}[h!]
 \labellist
\small\hair 2pt
 \pinlabel {$\delta_1$} [ ] at 20 182
 \pinlabel {$\delta_2$} [ ] at 93 84
 \pinlabel {$\delta_0$} [ ] at 68 68
 \pinlabel {$\delta_3$} [ ] at 82 206
 \pinlabel {$\delta_1$} [ ] at 20 57
\endlabellist
 	\centerline{\includegraphics[scale=.95]{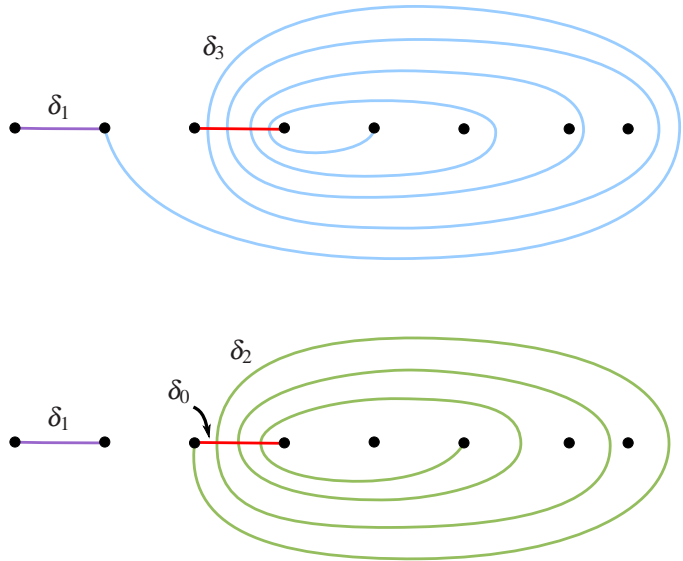}}
	\caption{Arcs giving an efficient geodesic in $\C(S_3)$ that is not tight}
	\label{F:vs}
\end{figure}

Finally we give examples of geodesics that are efficient but not tight.  Consider the arcs $\delta_0$, $\delta_1$, $\delta_2$, and $\delta_3$ shown in Figure~\ref{F:vs} (the arcs $\delta_0$, $\delta_1$, and $\delta_3$ are shown in the top picture of the figure and the arcs $\delta_0$, $\delta_1$, and $\delta_2$ are shown at the bottom).  Again, each $\delta_i$ represents a vertex $v_i$ of $\C(S_3)$ and again $d(v_0,v_3) = 3$ since $\delta_0$ and $\delta_3$ fill $X_{8}$.

To see that the oriented geodesic $v_0,v_1,v_2,v_3$ is initially efficient we notice that $\delta_1$ lies in a single region of $X_8$ determined by $\delta_0$ and $\delta_3$ and in that region it connects two marked points, one of which lies on $\delta_3$. It follows that the preimage of $\delta_1$ in $S_3$ is a single nonseparating simple closed curve and if we cut $S_3$ along the preimages of $\delta_0$ and $\delta_3$ then this nonseparating curve becomes a single diagonal in a single polygonal region of the cut-open surface.  From this it follows that $v_0,v_1,v_2,v_3$ is initially efficient.

A similar argument shows that the oriented geodesic $v_3,v_2,v_1,v_0$ is initially efficient.  Indeed, the intersection of the arc $\delta_2$ with each region of $X_8$ determined by $\delta_0$ and $\delta_3$ is a single arc.  It follows that the preimage of $\delta_2$ in $S_3$ intersects each polygonal region of $S_3$ in one or two arcs (depending on whether the corresponding arc in $X_8$ terminates at a marked point not contained in $\delta_0 \cup \delta_3$).  As such, any reference arc in $S_3$ for the preimages of $\delta_0$ and $\delta_3$ can intersect the preimage of $\delta_2$ in at most two points. The efficiency of $v_0,v_1,v_2,v_3$ follows.

We will now show that $v_0,v_1,v_2,v_3$ is not tight, in other words that $v_0,v_1,v_2,v_3$ is not contained in any tight multigeodesic.  Suppose $\sigma_0,\sigma_1,\sigma_2,\sigma_3$ were a tight multigeodesic containing $v_0,v_1,v_2,v_3$.  First of all, by definition we would have $\sigma_0 = v_0$ and $\sigma_3=v_3$.  Second, since (representatives of) $v_0$ and $v_2$ fill the complement of (a representative of) $v_1$ we must have that $\sigma_1=v_1$.  Now we notice that $v_2$ does not lie in a regular neighborhood of the union of representatives of $\sigma_1=v_1$ and $\sigma_3=v_3$ (since we can find an arc in $X_{8}$ that intersects $\delta_2$ without intersecting $\delta_1$ or $\delta_3$).  Therefore, for any choice of simplex $\sigma_2$ containing $v_2$ we will still have the property that $\sigma_2$ does not lie in a regular neighborhood of the union of representatives of $\sigma_1$ and $\sigma_3$; in particular, for any choice of $\sigma_2$ containing $v_2$, the sequence $\sigma_0,\sigma_1,\sigma_2,\sigma_3$ is not tight at $\sigma_2$.   Hence $v_0,v_1,v_2,v_3$ is not tight, as desired.  Again the generalization to higher genus is clear.
\end{proof}

%%%
%%%
%%%

\section{Existence of efficient paths}
\label{sec:simplify}

In this section we prove the main result of this paper, Theorem~\ref{T:simplify}.  The main point is to prove the existence of initially efficient geodesics (Proposition~\ref{P:initial}), and this will occupy most of the section.  At the end we give the additional inductive argument for the existence of efficient geodesics (Theorem~\ref{T:simplify}).  Let $g \geq 2$ be fixed throughout.

\subsection{Setup: a reducibility criterion}

Our first goal is to recast the problem of finding initially efficient paths in terms of sequences of numbers; see Proposition~\ref{prop:reducible} below.  

\p{Standard representatives and intersection sequences}
Let $v$ and $w$ be vertices of $\C(S_g)$ with $d(v,w)\geq 3$.   Let $v=v_0,\dots,v_n=w$ be an arbitrary path from $v$ to $w$.  We can choose representatives $\alpha_i$ of the $v_i$ with the following properties:
\begin{enumerate}
\item each $\alpha_i$ is in minimal position with both $\alpha_0$ and $\alpha_n$, 
\item each intersection $\alpha_i \cap \alpha_{i+1}$ is empty, and   
\item all triple intersections of the form $\alpha_i \cap \alpha_j \cap \alpha_k$ are empty.
\end{enumerate}
To do this, we take the $\alpha_i$ to be geodesics with respect to some hyperbolic metric on $S_g$ and then perform small isotopies to remove triple intersections.  We say that such a collection of representatives for the $v_i$ is \emph{standard}.  Note that we do not insist that $\alpha_i$ and $\alpha_j$ are in minimal position when $0 < i,j < n$ and $|i-j| > 1$.

Let $\gamma$ be a reference arc for the standard set of representatives $\alpha_0,\dots,\alpha_n$, by which we mean that:
\begin{enumerate}
 \item $\gamma$ has its interior disjoint from $\alpha_0 \cup \alpha_n$,
 \item $\gamma$ has endpoints disjoint from $\alpha_1,\dots,\alpha_{n-1}$,
 \item all triple intersections $\alpha_i \cap \alpha_j \cap \gamma$ are trivial for $i \neq j$, and
 \item $\gamma$ is in minimal position with each of $\alpha_1,\dots,\alpha_{n-1}$. 
\end{enumerate}
A reference arc for $\alpha_0,\dots,\alpha_n$ is automatically a reference arc for the triple $\alpha_0,\alpha_1,\alpha_n$ as in the introduction, but not the other way around.  We will need to deal with this discrepancy in the proof of Proposition~\ref{P:initial} below.

Denote the cardinality of $\gamma \cap (\alpha_1 \cup \cdots \cup \alpha_{n-1})$ by $N$.  Traversing $\gamma$ in the direction of some chosen orientation, we record the sequence of natural numbers $\sigma = (j_1,j_2,\dots,j_N) \in \{1,\dots,n-1\}^N$ so that the $i$th intersection point of $\gamma$ with $\alpha_1 \cup \cdots \cup \alpha_{n-1}$ lies in $\alpha_{j_i}$.  We refer to $\sigma$ as the \emph{intersection sequence} of the $\alpha_i$ along $\gamma$. 

\p{Complexity of paths and reducible sequences} We define the \emph{complexity} of an oriented path $v_0,\dots,v_n$ in $\C(S)$  to be
\[ \sum_{k=1}^{n-1} \left( i(v_0,v_k) + i(v_k,v_n) \right).\]
We say that a sequence $\sigma$ of natural numbers is \emph{reducible} under the following circumstances: whenever $\sigma$ arises as an intersection sequence for a (standard set of representatives for) path $v_0,\dots,v_n$ in $\C(S_g)$ there is another path $v_0',\dots,v_n'$ with $v_0'=v_0$ and $v_n'=v_n$ and with smaller complexity.
With this terminology in hand, the existence of initially efficient paths is a consequence of the following proposition.

\begin{proposition}
\label{prop:reducible}
Suppose $\sigma$ is a sequence of elements of $\{1,\dots,n-1\}$.  If $\sigma$ has more than $n-1$ entries equal to 1, then $\sigma$ is reducible.
\end{proposition}

We can deduce the existence of initially efficient geodesics easily from Proposition~\ref{prop:reducible}.  

\begin{proposition}
\label{P:initial}
Let $g \geq 2$.  If $v$ and $w$ are vertices of $\C(S_g)$ with $d(v,w) \geq 3$, then there exists an initially efficient geodesic from $v$ to $w$.  
\end{proposition}

\begin{proof}[Proof of Proposition~\ref{P:initial} assuming Proposition~\ref{prop:reducible}]

Let $v$ and $w$ be vertices of $\C(S_g)$ with $d(v,w) \geq 3$.  Since the complexity of any path from $v$ to $w$ is a natural number, there is a geodesic of minimal complexity.  We will show that any geodesic from $v$ to $w$ that has minimal complexity must be initially efficient.

To this end, we consider an arbitrary geodesic $v=v_0,\dots,v_n=w$ and we assume that it is not initially efficient.  In other words there is a set of representatives $\alpha_0,\alpha_1,\alpha_{n}$ for $v_0,v_1,v_{n}$ that are in minimal position and a reference arc $\gamma$ for $\alpha_0,\alpha_1,\alpha_{n}$ with $|\alpha_1 \cap \gamma| > n-1$.

We can extend the triple $\alpha_0,\alpha_1,\alpha_{n}$ to a set of standard representatives $\alpha_0,\dots,\alpha_{n}$ for the whole geodesic $v_0,\dots,v_n$.  What is more, we may assume that $\gamma$ is a reference arc for this full set of representatives $\alpha_0,\dots,\alpha_{n}$.  

Indeed, if $\gamma$ is not in minimal position with some $\alpha_i$ with $2 \leq i \leq n-1$ then by an adaptation of the usual bigon criterion for simple closed curves we have that $\gamma$ and $\alpha_i$ cobound an embedded bigon; if we choose an innermost such bigon (with respect to $\gamma$) and push the corresponding $\alpha_i$ across, then we can eliminate the bigon without creating any new points of intersection between $\gamma$ with any $\alpha_j$ or between any two $\alpha_j$.  (Alternatively, as in the introduction, we can assume that each $\alpha_i$ with $1 \leq i \leq n-1$ is a straight line segment in each polygon determined by $\alpha_0$ and $\alpha_{n}$ and we can take $\gamma$ to be any straight line segment; this procedure always yields a $\gamma$ that is in minimal position with each $\alpha_i$).

Since we did not change $\alpha_1$, the new intersection sequence of $\alpha_0,\dots,\alpha_{n}$ with $\gamma$ still has more than $n-1$ entries equal to 1.  By Proposition~\ref{prop:reducible}, the sequence $\sigma$ is reducible.  This implies that $v_0,\dots,v_n$ does not have minimal complexity, and we are done.
\end{proof}

Notice that the approach established in Proposition~\ref{prop:reducible} disregards all information about a path in $\C(S_g)$ except its intersection sequences.  For instance, we will not need to concern ourselves with how the strands of the $\alpha_i$ are connected outside of a neighborhood of $\gamma$.

We will prove Proposition~\ref{prop:reducible} in three stages.  First, in Section~\ref{section:stage 1} we describe a normal form for sequences of natural numbers (Lemma~\ref{L:sawtooth} below) and also describe an associated diagram for the normal form called the dot graph.   Next in Section~\ref{section:stage 2} we will show that if the dot graph exhibits certain geometric features---empty boxes and hexagons---then the sequence is reducible (Lemma~\ref{lemma:empty}).  Finally in Section~\ref{section:stage 3} we will show that any sequence in normal form that does not satisfy Proposition~\ref{prop:reducible} has a dot graph exhibiting either an empty box or an empty hexagon, hence proving Proposition~\ref{prop:reducible}.

\subsection{Stage 1: Sawtooth form and the dot graph}
\label{section:stage 1}

The main goal of this section is to give a normal form for sequences of natural numbers that interacts well with our notion of reducibility.  We also describe a way to diagram sequences in normal form called the dot graph.

\p{Sawtooth form} We say that a sequence $(j_1, j_2,\dots,j_k)$ of natural numbers is in \emph{sawtooth form}  if 
\[ j_i < j_{i+1}\ \ \Longrightarrow \ \ j_{i+1} = {j_i} +1.\]
An example of a sequence  in sawtooth form is $(1,2,2,3,4,3,4,3,4,2,3,4,5)$.  If a sequence of natural numbers is in sawtooth form, we may consider its \emph{ascending sequences}, which are the maximal subsequences of the form $k, k+1, \dots, k+m$.  In the previous example, the ascending sequences are $(1,2)$, $(2,3,4)$, $(3,4)$, $(3,4)$, and $(2,3,4,5)$.

\begin{lemma}
\label{L:sawtooth}
Let $\sigma$ be an intersection sequence.  There exists an intersection sequence $\tau$ in sawtooth form so that $\tau$ differs from $\sigma$ by a permutation of its entries and so that $\sigma$ is reducible if and only if $\tau$ is.  
\end{lemma}

\begin{proof}

Suppose $\sigma=(j_1, \dots, j_N)$ is the intersection sequence for a set of standard representatives $\alpha_0,\dots,\alpha_{n}$ along an arc $\gamma \subseteq \alpha_n\setminus\alpha_0$.  The basic idea we will use is that if $|j_i-j_{i+1}| > 1$, then we can modify $\alpha_{j_i}$ and $\alpha_{j_{i+1}}$ to new curves $\alpha_{j_i}'$ and $\alpha_{j_{i+1}}'$ so that the new curves still form a set of standard representatives for the same path and so that the new intersection sequence along $\gamma$ differs from $\sigma$ by a transposition of the consecutive terms $j_i$ and $j_{i+1}$; see  Figure~\ref{F:commute}.  We call this the resulting modification of $\sigma$ a commutation.

\begin{figure}[htbp!]
\labellist
\small\hair 2pt
\pinlabel $\gamma$ at -8 40
\pinlabel $\alpha_{j_i}$ at 20 95
\pinlabel $\alpha_{j_{i+1}}$ at 71 95
\pinlabel $\alpha_{j_i}'$ at 211 95
\pinlabel $\alpha_{j_{i+1}}'$ at 266 95
\endlabellist
\vspace*{3ex}
	\centerline{\includegraphics[width=.6\textwidth]{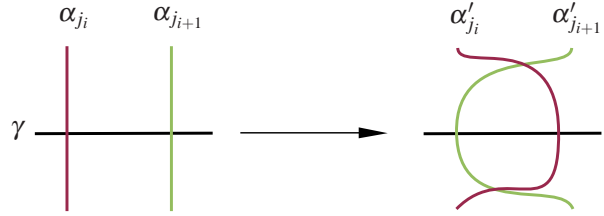}}
	\caption{A commutation}
	\label{F:commute}
\end{figure}

It suffices to show that if a sequence $\sigma$ is not in sawtooth form, then it is possible to perform a finite sequence of commutations so that the resulting sequence $\tau$ is in sawtooth form.  Indeed, the sequence $\tau$ appears as an intersection sequence for a particular path in $\C(S_g)$ if and only if $\sigma$ does (the key point is that commutations never result in a nonempty intersection of the form $\alpha_i \cap \alpha_{i+1}$).

We say that $\sigma$ fails to be in sawtooth form at the index $i$ if $j_{i+1} > j_i+1$.  Let $k=k(\sigma)$ be the highest index at which $\sigma$ fails to be in sawtooth form, and say that $k$ is zero if $\sigma$ is in sawtooth form.  Assuming $k > 0$, we will show that we can modify $\sigma$ by a sequence of commutations so that the highest index where the resulting sequence fails to be in sawtooth form is strictly less than $k$. 

We decompose $\sigma$ into a sequence of subsequences of $\sigma$, namely, 
\[ (\sigma_1,\sigma_2,\sigma_3,\sigma_4) \]
where $\sigma_2$ is the singleton $(j_k)$ and $\sigma_3$ is the longest subsequence of $\sigma$ starting from the $(k+1)$st term so that each term is greater than $j_k + 1$.  The sequences $\sigma_1$ and $\sigma_4$ are thus determined, and one or both might be empty.

By a series of commutations, we can modify $\sigma$ to the sequence
\[ \sigma' = (\sigma_1,\sigma_3,\sigma_2,\sigma_4). \]
We claim that $k(\sigma') < k(\sigma)$.  Since the length of $\sigma_1$ is $k-1$, it is enough to show that the subsequence $(\sigma_3,\sigma_2,\sigma_4)$ is in sawtooth form.

By the definition of $k$, we know that $\sigma_3$ is in sawtooth form.  Next, the last term of $\sigma_3$ is greater than $j_k+1$ and the first (and only) term of $\sigma_2$ is $j_k$, and so these terms satisfy the definition of sawtooth form.  We know $\sigma_2 = (j_k)$ and the first term of $\sigma_4$, call it $j$, is at most $j_k+1$, and so these terms are also in sawtooth form.  Finally, the subsequence $\sigma_4$ is in sawtooth form by the definition of $k$.  This completes the proof.
\end{proof}

\p{Dot graphs}  It will be useful to draw the graph in $\R_{\geq 0}^2$ of a given sequence of natural numbers, where the sequence is regarded as a function $\{1,\dots,N\} \to \N$.  The points of the graph of a sequence $\sigma$ will be called \emph{dots}.  We decorate the graph by connecting the dots that lie on a given line of slope 1; these line segments will be called \emph{ascending segments}.  The resulting decorated graph will be called the \emph{dot graph} of $\sigma$ and will be denoted $G(\sigma)$; see Figure~\ref{F:dot graph}.

\begin{figure}[htbp!]
\labellist
\small\hair 2pt
\endlabellist
	\centerline{\includegraphics[width=.7\textwidth]{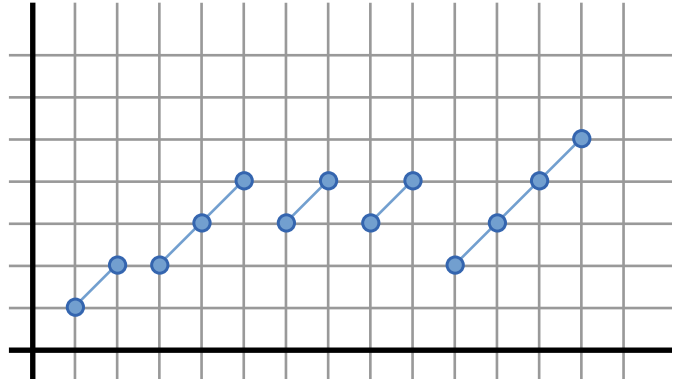}}
	\caption{Example of dot graph of a sequence in sawtooth form}
	\label{F:dot graph}
\end{figure}

%%%
%%%
%%%

\subsection{Stage 2: Dot graph polygons and surgery}
\label{section:stage 2}

The goal of this section is to describe certain geometric shapes than can arise in a dot graph, and then to prove that if the dot graph $G(\sigma)$ admits one of these shapes then the sequence $\sigma$ is reducible (Lemma~\ref{lemma:empty}).

\p{Dot graph polygons} 
 We say that a polygon in the plane is a \emph{dot graph polygon} if
\begin{enumerate}
\item the edges all have slope 0 or 1,
\item the edges of slope 0 have nonzero length, and 
\item the vertices all have integer coordinates.
\end{enumerate}
The edges of slope 1 in a dot graph polygon are called \emph{ascending edges} and the edges of slope 0 are called \emph{horizontal edges}.

Let $\sigma$ be a sequence of natural numbers in sawtooth form.  A dot graph polygon is a \emph{$\sigma$-polygon} if:
\begin{enumerate}
 \item the vertices are dots of $G(\sigma)$ and
 \item the ascending edges are contained in ascending segments of $G(\sigma)$. 
\end{enumerate}

\begin{figure}[htbp!]
	\centerline{\includegraphics[width=1.1\textwidth]{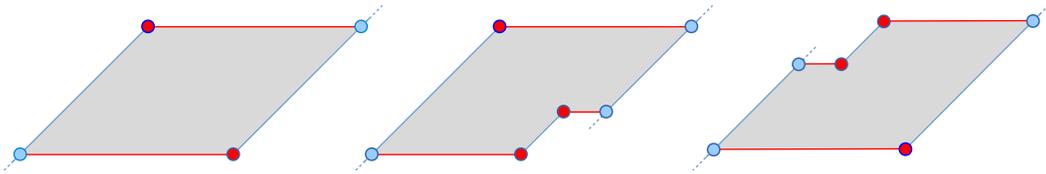}}
	\caption{A box, a hexagon of type 1, and a hexagon of type 2; the red (darker) dots are required to be endpoints of ascending segments, while the blue (lighter) dots may or may not be endpoints}
	\label{F:types}
\end{figure}

A \emph{box} in $G(\sigma)$ is a $\sigma$-quadrilateral $P$ with the following two properties:
\begin{enumerate}
 \item the leftmost ascending edge contains the highest point of some ascending segment of $G(\sigma)$ and
 \item the rightmost ascending edge contains the lowest point of some ascending segment of $G(\sigma)$.
\end{enumerate}

We will also need to deal with hexagons.  Up to translation and changing the edge lengths, there are four types of dot graph hexagons; two have an acute exterior angle, and we will not need to consider these.  Notice that a dot graph hexagon necessarily has a leftmost ascending edge, a rightmost ascending edge, and a middle ascending edge.  This holds even for degenerate hexagons since horizontal edges are required to have nonzero length.    

A \emph{hexagon of type 1} in $G(\sigma)$ is a $\sigma$-hexagon where:
\begin{enumerate}
\item no exterior angle is acute, 
\item the middle ascending edge is an entire ascending segment of $G(\sigma)$, and 
\item the minimum of the middle ascending edge equals the minimum of the leftmost ascending edge, 
\item the leftmost ascending edge contains the highest point of an ascending segment of $G(\sigma)$.
\end{enumerate}
Similarly, a \emph{hexagon of type 2} in $G(\sigma)$ is a $\sigma$-hexagon that satisfies the first two conditions above and the following third  and fourth conditions:
\begin{enumerate}
\item[($3'$)] the maximum of the middle ascending edge equals the maximum of the rightmost ascending edge,
\item[($4'$)]  the rightmost ascending edge contains the lowest point of an ascending segment of $G(\sigma)$.
\end{enumerate}

See Figure~\ref{F:types} for pictures of boxes and hexagons of types 1 and 2.

\medskip

The following lemma is the main goal of this section.  We say that a horizontal edge of a $\sigma$-polygon is \emph{pierced} if its interior intersects $G(\sigma)$.  Also, we say that a $\sigma$-polygon is \emph{empty} if it there are no points of $G(\sigma)$ in its interior.

\begin{lemma}
\label{lemma:empty}
Suppose that $\sigma$ is a sequence of natural numbers in sawtooth form and that $G(\sigma)$ has an empty, unpierced box or an empty, unpierced  hexagon of type 1 or 2.  Then $\sigma$ is reducible.  
\end{lemma}

Before we prove Lemma~\ref{lemma:empty}, we need to introduce another topological tool, surgery on curves.

\p{Surgery}  Let $\alpha$ be a simple closed curve in a surface and let $\gamma$ be an oriented arc so that $\alpha$ and $\gamma$ are in minimal position.  We can form a new curve $\alpha'$ from $\alpha$ by performing surgery along $\gamma$ as follows.  We first remove from $\alpha$ small open neighborhoods of two points of $\alpha \cap \gamma$ that are consecutive along $\gamma$.  What remains of $\alpha$ is a pair of arcs; we can connect the endpoints of either arc by another arc $\delta$ that lies in a small neighborhood of $\gamma$ in order to create the new simple closed curve $\alpha'$ (the other arc of $\alpha$ is discarded); see Figure~\ref{F:surgery}. 

We draw a neighborhood of $\gamma$ in the plane so that $\gamma$ is a horizontal arc oriented to the right.   We say that $\alpha'$ is obtained from $\alpha$ by $++$, $+-$, $-+$, or $--$ surgery along $\gamma$; the first symbol is $+$ or $-$ depending on whether the first endpoint of $\delta$ (as measured by the orientation of $\gamma$) lies above $\gamma$ or below, and similarly for the second symbol.

In general, for a given pair of intersection points of a curve $\alpha$ with $\gamma$, exactly two of the four possible surgeries result in a simple closed curve.  If we orient $\alpha$, then the two intersection points of $\alpha$ with $\gamma$ can either agree or disagree.  If they agree, then the $+-$ and $-+$ surgeries, the \emph{odd surgeries}, result in a simple closed curve, and if they disagree, the $++$ and $--$ surgeries, the \emph{even surgeries}, result in a simple closed curve.  

\begin{figure}[h!]
\labellist
\small\hair 2pt
\pinlabel $\alpha$ at 16 96
\pinlabel $\alpha$ at 74 96
\pinlabel $\gamma$ at 46 31
\pinlabel \fbox{\tiny $++$} at 185 105
\pinlabel \fbox{\tiny $--$} at 283 105
\pinlabel \fbox{\tiny $+-$} at 376 105
\pinlabel \fbox{\tiny $-+$} at 472 105
\endlabellist \vspace*{.25in}
	\centerline{\includegraphics[width=\textwidth]{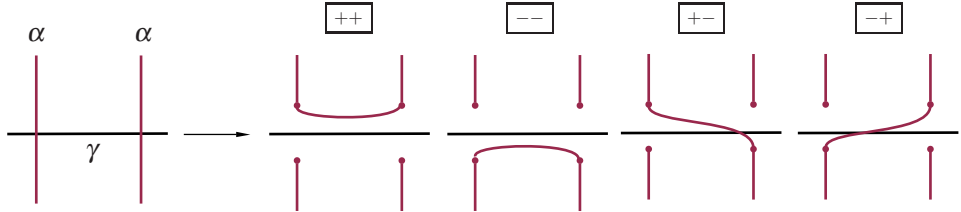}}
	\caption{The four types of surgery on a curve along an arc}
	\label{F:surgery}
\end{figure}

These surgeries will of course only be of use to us if the curve $\alpha'$ is an essential simple closed curve in $S$.  One variant of the well-known bigon criterion is that a curve $\alpha$ and an arc $\gamma$ are in minimal position if and only if every closed curve formed from $\alpha$ and $\gamma$ as above is essential.  Indeed, the proof in the case where $\alpha$ and $\gamma$ are both curves (see \cite[Proposition 3.10]{FLP}) can be adapted to this case.  Thus our $\alpha'$ is essential.

\p{Surfaces with boundary}\label{fail}  In order to show that our surgered curves are essential, we used a version of the bigon criterion.  This bigon criterion is exactly what fails in the case of surfaces with boundary.  For instance, suppose that the surface $S$ has at least two boundary components and consider a simple closed curve $\alpha$ that cuts off a pair of pants in $S$.  If $\gamma$ is an arc that intersects $\alpha$ in two points then both of the curves obtained by surgering $\alpha$ along $\gamma$ are homotopic to components of the boundary of $S$, neither of which represents a vertex of $\C(S)$.

\bigskip

We now use the surgeries described above to prove that a dot graph with an empty, unpierced box or an empty, unpierced hexagon of type 1 or 2 corresponds to a sequence that is reducible.

\begin{proof}[Proof of Lemma~\ref{lemma:empty}]

Suppose that $\sigma$ appears as an intersection sequence for a reference arc $\gamma$ for a set of standard representatives $\alpha_0, \dots, \alpha_n$ for a path $v_0,\dots,v_n$ in $\C(S_g)$.  We need to replace the $\alpha_{i}$ with new curves $\alpha_{i}'$ so that the resulting path from $v_0$ to $v_n$ has smaller complexity.  We treat the three cases in turn, according to whether $G(\sigma)$ has an empty, unpierced box or an empty, unpierced hexagon of type 1 or 2.  

Suppose $G(\sigma)$ has an empty, unpierced box $P$.  By the definitions of sawtooth form and empty boxes there are no ascending edges of $G(\sigma)$ in the vertical strip between the two ascending edges of $P$, that is, the dots of $P$ correspond to a consecutive sequence of intersections along $\gamma$:
\[ \alpha_k, \dots, \alpha_{k+m}, \ \ \alpha_k, \dots, \alpha_{k+m} \]
where $1 \leq k \leq k+m \leq n-1$.

\begin{figure}[htbp!]
\labellist
\small\hair 2pt
\pinlabel $3$ at 39 224
\pinlabel $4$ at 69 224
\pinlabel $5$ at 99 224
\pinlabel $3$ at 125 224
\pinlabel $4$ at 157 224
\pinlabel $5$ at 187 224
\pinlabel $3'$ at 347 227
\pinlabel $4'$ at 380 227
\pinlabel $5'$ at 494 227
\pinlabel $5$ at 160 78
\pinlabel {\tiny $-+$} at 285 84
\pinlabel $4$ at 160 52
\pinlabel {\tiny $++$} at 265 58
\pinlabel $3$ at 160 27
\pinlabel {\tiny $+-$} at 244 32
\endlabellist
	\centerline{\includegraphics[width=1.0\textwidth]{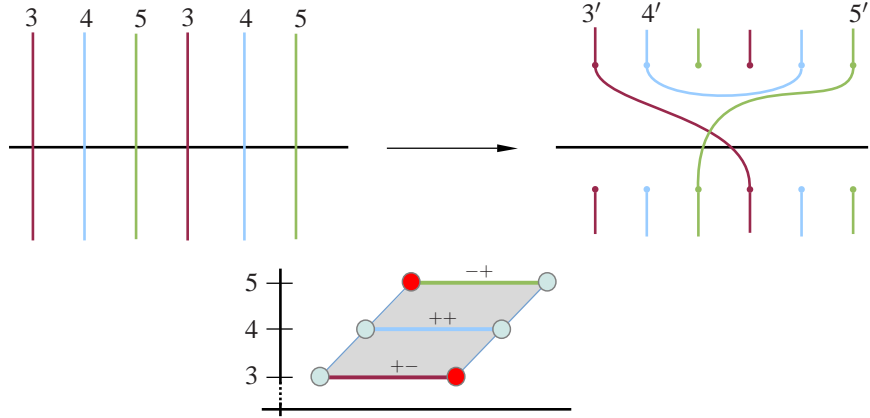}}
	\caption{{\small An example of a set of surgeries as in the box case of Lemma~\ref{lemma:empty}}}
	\label{F:repetition}
\end{figure}

First, for $i \notin \{k , \dots, k + m\}$ we set $\alpha_i'=\alpha_i$.    We then define $\alpha_k',\dots, \alpha_{k+m}'$ inductively: for $i=k,\dots, k+m$, the curve $\alpha_{i}'$ is obtained by performing surgery along $\gamma$ between the two points of $\alpha_{i} \cap \gamma$ corresponding to dots of $P$ and the surgeries are chosen so that they form a path in the directed graph in Figure~\ref{F:digraph} (of course for each $i$ we must choose one of the two surgeries that results in a closed curve).

\begin{figure}[h!]
\labellist
\small\hair 2pt
\pinlabel \fbox{\tiny $++$} at 65 108
\pinlabel \fbox{\tiny $+-$} at 32 32
\pinlabel \fbox{\tiny $-+$} at 130 75
\pinlabel \fbox{\tiny $--$} at 98 0
\endlabellist
	\centerline{\includegraphics{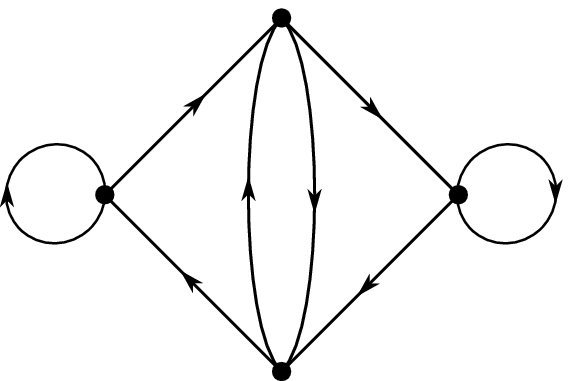}}
	\caption{The directed graph used in the proof of Lemma~\ref{lemma:empty}}
\label{F:digraph}
\end{figure}

The vertices of the graph in Figure~\ref{F:digraph} correspond to the four types of surgeries: $++$, $+-$, $-+$, and $--$, and the rule is that the second sign of the origin of a directed edge is the opposite of the first sign of the terminus.  Since every vertex has one outgoing arrow pointing to an even surgery and one outgoing arrow pointing to an odd surgery, the desired sequence of surgeries exists; in fact it is completely determined by the choice of surgery on $\alpha_k$, and so there are exactly two possible sequences.  See Figure~\ref{F:repetition} for an example of this procedure; there we perform $+-$ surgery on $\alpha_{3}$, then $++$ surgery on $\alpha_{4}$, then $-+$ surgery on $\alpha_{5}$.  

For $0 \leq i \leq n$, let $v_i'$ be the vertex of $\C(S_g)$ represented by $\alpha_i'$.  We need to check that the $v_i'$ certify the reducibility of $\sigma$, namely that
\begin{enumerate}
 \item  $v_0'=v_0$ and $v_n' = v_n$,
 \item each $v_i'$ is connected to $v_{i+1}'$ by an edge in $\C(S_g)$, and 
 \item  the complexity of $v_0',\dots,v_n'$ is strictly smaller than that of $v_0,\dots,v_n$.
\end{enumerate}
The first condition holds because $1 \leq k \leq k+m \leq n-1$.  The second condition holds because each intersection $\alpha_i \cap \alpha_{i+1}$ is empty and the surgeries do not create new  intersections.    For the third condition, we claim that something stronger is true, namely, that
\[ i(v_0,v_i') + i(v_i',v_n) \leq i(v_0,v_i) + i(v_i,v_n) \]
for all $i$ and that
\[ i(v_0,v_k') + i(v_k',v_n) < i(v_0,v_k) + i(v_k,v_n). \]
Indeed, if we consider the polygonal decomposition of $S_g$ determined by $\alpha_0 \cup \alpha_n$ we see that when we surger two strands of some $\alpha_i$ along $\gamma$ we create no new intersections with $\alpha_0 \cup \alpha_n$ and we remove two intersections with $\alpha_0 \cup \alpha_n$ (we might also create a bigon, but this would only help our case).  Since we performed at least one surgery---on $\alpha_k$---our claim is proven.  

\medskip

\begin{figure}[htbp!]
\labellist
\small\hair 2pt
\pinlabel $3$ at 31 278
\pinlabel $4$ at 45 278
\pinlabel $5$ at 59 278
\pinlabel $6$ at 73 278
\pinlabel $7$ at 87 278
\pinlabel $3$ at 101 278
\pinlabel $4$ at 115 278
\pinlabel $5$ at 129 278
\pinlabel $2$ at 144 278
\pinlabel $3$ at 158 278
\pinlabel $4$ at 172 278
\pinlabel $5$ at 186 278
\pinlabel $6$ at 200 278
\pinlabel $7$ at 214 278
\pinlabel $5'$ at 360 278
\pinlabel $6'$ at 374 278
\pinlabel $3'$ at 402 278
\pinlabel $7'$ at 516 278
\pinlabel \begin{rotate}{-20}\textcolor{red}{\Large $\boldsymbol{\times}$}\end{rotate}  at 475 262
\pinlabel $4'$ at 374 172
\pinlabel {\tiny $\mplus$} at 292 119
\pinlabel $7$ at 122 114
\pinlabel {\tiny $++$} at 280 102
\pinlabel $6$ at 122 97
\pinlabel {\tiny $\pminus$} at 229 85
\pinlabel $5$ at 122 79
\pinlabel \begin{rotate}{-20}\textcolor{red}{\Large $\boldsymbol{\times}$}\end{rotate} at 327 76
\pinlabel {\tiny $\mminus$} at 217 68
\pinlabel $4$ at 122 63
\pinlabel {\tiny $\mplus$} at 198 51
\pinlabel $3$ at 122 47
\pinlabel $2$ at 122 32
\endlabellist
	\centerline{\includegraphics[width=1.0\textwidth]{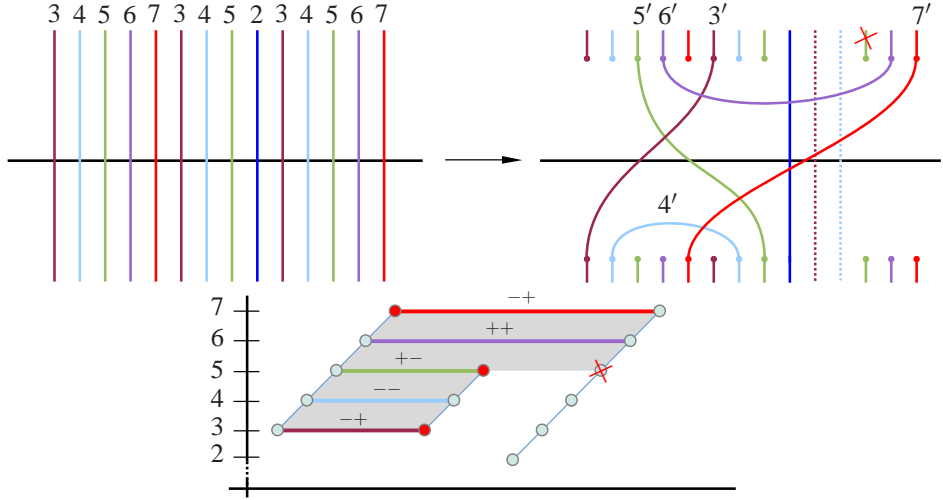}}
	\caption{{\small An example of a set of surgeries as in the hexagon case of Lemma~\ref{lemma:empty}}}
	\label{F:glock surgery}
\end{figure}

The cases of empty, unpierced hexagons of types 1 and 2 are similar, but one new idea is needed.  These two cases are almost identical, and so we will only treat the first case, that is, we suppose $G(\sigma)$ has an empty, unpierced hexagon $P$ of type 1.  By the definition of sawtooth form and the definition of an empty, unpierced hexagon of type 1, there are no ascending segments of $G(\sigma)$ in the vertical strip between the leftmost and middle ascending edges of $P$ and any ascending segments of $G(\sigma)$ that lie in the vertical strip between the middle and rightmost ascending segments have their highest point strictly below the lower-right horizontal edge of $P$.  It follows that the dots of $P$ correspond to a sequence of intersections along $\gamma$ of the following form
\[ \alpha_k, \dots, \alpha_{k+m}, \ \ \alpha_k, \dots, \alpha_{k+\ell},
\ \ \alpha_{j_1}, \dots, \alpha_{j_p}, \ \   \alpha_{k+\ell}, \dots, \alpha_{k+m} \]
where $1 \leq k \leq k+\ell \leq k+m \leq n-1$, $p \geq 0$, and each $j_i <  \alpha_{k+\ell}$.  See Figure~\ref{F:glock surgery} for an example where $k=3$, $\ell=2$, $m=4$, and $p=0$.

Again, for each $\alpha_i$ with $i \notin \{k , \dots, k + m\}$ we set $\alpha_i'=\alpha_i$.  Each of the remaining $\alpha_i$ corresponds to exactly two dots in $P$ except for $\alpha_{k+\ell}$, which corresponds to three.  Let $\alpha_{k+\ell}'$ be the curve obtained from $\alpha_{k+\ell}'$ via surgery along $\gamma$ between the first two (leftmost) points of $\alpha_{k+\ell}' \cap \gamma$ corresponding to dots of $P$ and satisfying the following property: $\alpha_{k+\ell}'$ does not contain the arc of $\alpha_{k+\ell}$ containing the third (rightmost) point of $\alpha_{k+\ell} \cap \gamma$ corresponding to a dot of $P$.  As always, there are two choices of surgery given two consecutive points of $\alpha_{k+\ell} \cap \gamma$; one contains this third intersection point and one does not.

We then define $\alpha_{k+\ell-1}', \dots, \alpha_k'$ inductively as before using the diagram above (notice the reversed order), and finally we define $\alpha_{k+\ell+1}',\dots, \alpha_{k+m}'$ inductively as before.  

By our choice of $\alpha_{k+\ell}'$, we have that $\alpha_{k+\ell}' \cap \alpha_{k+\ell+1}' = \emptyset$, as required; indeed, we eliminated the strand of $\alpha_{k+\ell}'$ that was in the way between the two strands of $\alpha_{k+\ell+1}$ being surgered.  Also, since each $j_i$ is strictly less than $k+\ell$, the curves $\alpha_{k+\ell+1}',\dots, \alpha_{k+m}'$ satisfy the condition that $\alpha_i' \cap \alpha_{i+1}' = \emptyset$.  The other conditions in the definition of a reducible sequence are easily verified as before.  This completes the proof of the lemma.
\end{proof}

%%%
%%%
%%%

\subsection{Stage 3: Innermost polygons}
\label{section:stage 3}

In this section we will put together Lemmas~\ref{L:sawtooth} and~\ref{lemma:empty} in order to prove Proposition~\ref{prop:reducible}.  We begin with two lemmas.

\begin{lemma}
\label{lemma:almost box}
If a dot graph $G(\sigma)$ contains a box $P$ pierced in exactly one edge, then it contains an unpierced box.
\end{lemma}

\begin{proof}

Denote the ascending edges of $P$ by $e$ and $f$.  There is an ascending segment $e'$ intersecting the interior of exactly one of the two horizontal edges of $P$; we choose $e'$ to be rightmost if it intersects the bottom edge of $P$ and leftmost if it intersects the top edge.  Either way, we find a box $P'$ pierced in at most one edge and where one ascending edge is contained in $e'$ and the other ascending edge is contained in $P$.  The box $P'$ has horizontal edges strictly shorter than those of $P$.  Therefore, we may repeat the process until it eventually terminates, at which point we find the desired unpierced box.
\end{proof}

\begin{lemma}
\label{lemma:innermost}
Among all unpierced boxes and hexagons of type 1 and 2 in a dot graph $G(\sigma)$, an innermost unpierced box or hexagon of type 1 or 2 is empty.
\end{lemma}

\begin{proof}

We treat the three cases separately.  First suppose that $P$ is an unpierced box that is not empty.  We will show that $P$ either contains another unpierced box or an unpierced hexagon of type 1.  Let $e$ be an ascending segment contained in the interior of $P$.  We choose $e$ so that $\max(e)$ is maximal among all such ascending segments, and we further choose $e$ to be rightmost among all ascending segments with maximum equal to $\max(e)$.  

\begin{figure}[h!]
	\centerline{\includegraphics[width=1\textwidth]{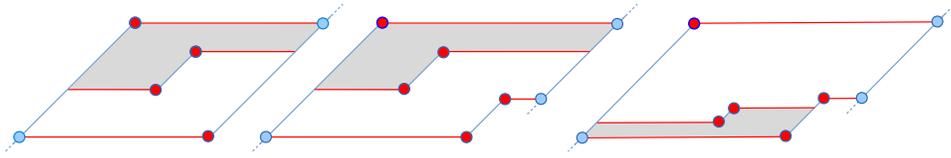}}
	\caption{Inside a box, inside a hexagon, inside a hexagon}
	\label{F:inside}
\end{figure}

There is a unique (possibly degenerate) hexagon $P'$ of type 1 with one edge equal to $e$, and the other two edges contained in the ascending edges of $P$; see the left-hand side of Figure~\ref{F:inside}.  If $P'$ is unpierced, we are done, so assume that 
$P'$ is pierced.  By construction, the top horizontal edge of $P'$ and the lower-right horizontal edge of $P'$ are unpierced.
Suppose that the interior of the lower-left horizontal edge of $P'$ were pierced.  Let $e'$ be the rightmost ascending segment of $G(\sigma)$ that pierces this edge of $P'$.  By the choice of $e$, we have that $\max(e') \leq \max(e)$, and so there is a box pierced in at most one edge whose ascending edges are contained in $e'$ and $e$.  By Lemma~\ref{lemma:almost box},  there is an unpierced box contained in this pierced box, and so $P$ is not innermost.

\medskip

The second case is where $P$ is an unpierced hexagon of type 1.  Again suppose that $P$ is not empty.  Let $e$ be an ascending segment contained in the interior of $P$ that has the largest maximum $\max(e)$ over all such segments and is rightmost among all such ascending segments.  Let $m$ denote the middle ascending edge of $P$.  It follows from the fact that $\sigma$ is in sawtooth form that there are no ascending segments of $G(\sigma)$ that lie inside $P$ and to the right of $m$; so $e$ lies to the left of $m$.  We now treat two subcases, depending on whether $\max(e) > \max(m)$ or not.

If $\max(e) > \max(m)$, there is a maximal hexagon $P'$ of type 1 with ascending edges contained in $P \cup e$ as in the middle picture of Figure~\ref{F:inside}.  By the same argument as in the previous case, $P'$ is either unpierced or it contains an unpierced box.  

If $\max(e) \leq \max(m)$, the argument is similar.  There is a hexagon $P'$ of type 2 as shown in the right-hand side of Figure~\ref{F:inside}.  The topmost edge of $P'$ is unpierced by the choice of $e$.  The bottom edge of $P'$ is unpierced since it is a horizontal edge for $P$, which is unpierced.  And if the third horizontal edge of $P'$ were pierced, we could find a box pierced in at most one edge, hence an unpierced box, as in the previous cases.  It follows that $P'$ is unpierced and again  $P$ is not innermost.

\medskip

The third and final case is where $P$ is an unpierced hexagon of type 2.  This is completely analogous to the previous case; in fact, if we rotate the two pictures from the type 1 case by $\pi$ we obtain the required pictures for the type 2 case. 
\end{proof}

We can now use the two previous lemmas to prove Proposition~\ref{prop:reducible}.

\begin{proof}[Proof of Proposition~\ref{prop:reducible}]

Let $\sigma$ be a sequence of elements of $\{1,\dots,n-1\}$.  By Lemma~\ref{L:sawtooth} we may assume that $\sigma$ is in sawtooth form without changing the number of entries equal to 1; call this number $k$.  Let $e_1, \dots e_k$ denote the ascending segments of $G(\sigma)$ with minimum equal to 1, ordered from left to right.

If $\max(e_{i+1}) < \max(e_i)$ for all $i$, then since $\max(e_1) \leq n-1$ it follows that $k \leq n-1$.  Therefore, it suffices to show that if $\max(e_{i+1}) \geq \max(e_i)$ for some $i$ then $\sigma$ is reducible.

Suppose then that $\max(e_{i+1}) \geq \max(e_i)$ for some $i$.  The first step is to show that $G(\sigma)$ has an unpierced box.  Let $e$ be the first ascending segment (from left to right) that appears after $e_i$ and has $\max(e) \geq \max(e_i)$.  Because $\min(e) \geq \min(e_i) = 1$, there is evidently a (possibly degenerate) box $P$ with two edges contained in $e_i$ and $e$ and two horizontal edges with heights $\min(e)$ and $\max(e_i)$.  By the definition of $e$, the interior of the upper horizontal edge of $P$ is disjoint from $G(\sigma)$, so $P$ is pierced in at most one edge.  By Lemma~\ref{lemma:almost box}, $P$ contains an unpierced box.

Let $P$ now be an innermost unpierced box or hexagon of type 1 or 2; such $P$ exists because each $\sigma$-polygon contains a finite number of dots of $G(\sigma)$ and a polygon contained inside another polygon contains a fewer number of dots.  By Lemma~\ref{lemma:innermost}, the polygon $P$ is empty.  By Lemma~\ref{lemma:empty}, $\sigma$ is reducible.
\end{proof}

%%%
%%%
%%%

\subsection{From initially efficient geodesics to efficient geodesics} At this point we have established the existence of initially efficient geodesics (Proposition~\ref{P:initial}).  It remains to establish the existence of efficient geodesics (Theorem~\ref{T:simplify}).  

\p{Total complexity} For an oriented path $q$ in $\C(S_g)$ with vertices $w_0,\dots,w_n$ define the complexity $\kappa(q)$ as before:
\[ \kappa(q) = \sum_{k=1}^{n-1} \left(i(w_0,w_k) + i(w_k,w_n)\right).\]
Next, for an oriented path $p$ with vertices $v_0,\dots,v_n$, let $p_1$ be the oriented path $v_n,\dots,v_{n-3}$ and let $p_k$ be the oriented path $v_{n-k-1},\dots,v_n$ for $2 \leq k \leq n-1$.  
We will relabel the vertices of $p_k$ as $w_0,\dots,w_{n_k}$.  
The \emph{total complexity} of a path $p$ is the ordered $(n-1)$-tuple:
\[ \hat \kappa(p) = (\kappa(p_1),\dots,\kappa(p_{n-1})).\]
We order the set $\N^{n-1}$---hence the set of total complexities---lexicographically.

\begin{proof}[Proof of Theorem~\ref{T:simplify}]

Let $v$ and $w$ be vertices of $\C(S_g)$ with $d(v,w) \geq 3$.  We claim that any geodesic from $v$ to $w$ that has minimal total complexity must be efficient.

Let $p$ be an arbitrary geodesic $v=v_0,\dots,v_n=w$ and assume that $p$ is not efficient.  In other words, one of the corresponding paths $p_k$ with vertices $w_0,\dots,w_{n_k}$ is not initially efficient.  This is the same as saying that there is a set of representatives $\beta_0,\beta_1,\beta_{n_k}$ for $w_0,w_1,w_{n_k}$ that are in minimal position and a reference arc $\gamma$ with $|\beta_1 \cap \gamma| > n_k-1$.

As in the proof of Proposition~\ref{P:initial} we can extend the triple $\beta_0,\beta_1,\beta_{n_k}$ to a full standard set of representatives $\beta_0,\dots,\beta_{n_k}$ for $p_k$.  And as in that proof there are surgeries that reduce the complexity of $p_k$.  The curves obtained by these surgeries not only give a new path between the endpoints of $p_k$, but they also give rise to a new path between $v$ and $w$.

The key observation here is that, by our choice of the order of the $p_i$, the surgeries used in modifying $p_k$ do not increase the complexity of any $p_i$ with $i < k$.  Indeed, these surgeries do not increase the intersection between any of the curves $\beta_0,\dots,\beta_{n_k}$ and all of the vertices of $p$ used in the computation of $\kappa(p_i)$ with $i < k$ are already vertices of $p_k$, namely, the vertices represented by $\beta_0,\dots,\beta_{n_k}$.  The theorem follows.
\end{proof}

\subsection{An improved algorithm in a special case} We end this section by stating and proving the alternate version of the efficient geodesic algorithm that was used in the example at the start of Section~\ref{subsec:example}.  This proposition is equivalent to the main theorem (Theorem 1.1) of the first version of this paper \cite{v1}.  

\begin{proposition}
 \label{P:improvement}
Suppose $v$ and $w$ are vertices of $\C(S_g)$ with $d(v,w) \geq 3$.  Let $\alpha$ and $\beta$ be representatives of $\alpha$ and $\beta$ that are in minimal position.  Then there is a geodesic $v=v_0,\dots,v_n=w$ and a representative $\alpha_1$ of $v_1$ so that the number of intersections of $\alpha_1$ with each arc of $\beta\setminus \alpha$ is at most $d(v,w)-2$.
\end{proposition}

\begin{proof}

The proof is essentially the same as the proof of Theorem~\ref{T:simplify}.  The only added observation is that, since $\gamma$ is a subset of $\beta$, every intersection sequence can be taken to have entries in $\{1,\dots,n-2\}$ instead of $\{1,\dots,n-1\}$.  
\end{proof}

Note that in the special case that vertices $v$ and $w$ have representatives $\alpha$ and $\beta$ that cut the surface into rectangles and hexagons only (e.g. the example of Section~\ref{S:computation}), then every reference arc is parallel to a reference arc as in Proposition~\ref{P:improvement}, and so in this case there are geodesics that are extra efficient in the sense that the intersection of a representative of $v_1$ with any reference arc is at most $n-2$ instead of $n-1$.

%%%
%%%
%%%

\appendix

\section{Webb's algorithm}
\label{sec:Webb}

In this appendix we give an exposition of Webb's algorithm for computing distance in $\C(S)$.  As with the efficient geodesic algorithm we will make the inductive hypothesis that for some $n \geq 2$ we have an algorithm to determine if the distance between two vertices is $0, \dots,n-1$ and we would like to give an algorithm for determining if the distance between two vertices is $n$.  First we introduce an auxiliary tool, the arc complex for a surface with boundary.

\p{Arc complex} Let $F$ be a compact surface with nonempty boundary.  The arc complex $\A(F)$ is the simplicial complex with $k$-simplices corresponding to $(k+1)$-tuples of homotopy classes of essential arcs in $F$ with pairwise disjoint representatives.  Here, homotopies are allowed to move the endpoints of an arc along $\partial F$, and an arc is essential if it is not homotopic into $\partial F$.  

\p{The algorithm} A maximal simplex of $\A(F)$ can be regarded as a triangulation of the surface obtained from $F$ by collapsing each component of the boundary to a point.  If $F$ is a compact, orientable surface of genus $g$ with $m$ boundary components, then the number of edges in any such triangulation is $6g+3m-6$.

Let $v$ and $w$ be two vertices of $\C(S)$ with $d(v,w) \geq 3$.  As in the efficient geodesic algorithm, it suffices by the induction hypothesis to list all candidates for vertices $v_1$ on a tight geodesic $v=v_0,\dots,v_{n}=w$.  Since there are finitely many vertices in each simplex of $\C(S)$ it further suffices to list all candidates for simplices $\sigma_1$ on a tight multigeodesic $v=\sigma_0,\dots,\sigma_{n}=w$.  

Suppose we have such a tight multigeodesic $v=\sigma_0,\dots,\sigma_{n}=w$.  We can choose representatives $\alpha_i$ of the $\sigma_i$ so that $\alpha_i \cap \alpha_{i+1} = \emptyset$ for all $i$ and so that each $\alpha_i$ lies in minimal position with $\alpha_0$.   If we cut $S$ along $\alpha_0$, we obtain a compact surface $S'$, some of whose boundary components correspond to $\alpha_0$.  

For each $i > 1$, the representative $\alpha_i$ gives a collection of disjoint arcs in $S'$ and hence a simplex $\tau_i$ of $\A(S')$ (some arcs of $\alpha_i$ might be parallel and these get identified in $\A(S')$).  For $i \geq 3$, the collection of arcs is filling, which means that when we cut $S'$ along these arcs we obtain a collection of disks and boundary-parallel annuli, and we say that the corresponding simplex of $\A(S')$ is filling.

Since there is a unique configuration for $\alpha_{n}$ and $\alpha_0$ in minimal position, there is a unique possibility for $\tau_{n}$.   As $\tau_{n} \cup \tau_{n-1}$ is contained in a simplex of the arc complex of $S'$ and since $\tau_{n}$ is filling, there are finitely many possibilities for $\tau_{n-1}$ (and we can explicitly list them).  This is the key point: there are infinitely many vertices of $\C(S)$ that correspond to any given simplex in the arc complex, but there are finitely many choices for the simplex itself.

Because $\tau_i$ is filling whenever $i \geq 3$, we can continue this process inductively, and explicitly list all possibilities for $\tau_2$.  Now, by the definition of a tight multigeodesic in $\C(S)$, the simplex $\sigma_1$ is represented by the union of the essential components of the boundary of a regular neighborhood of $\alpha_0 \cup \alpha_2$.  Equivalently, any such $\sigma_1$ is given by a regular neighborhood of the union of $\partial S'$ with a representative of $\tau_2$.  Hence there are finitely many (explicitly listable) possibilities for $\sigma_1$, as desired.

\p{A bound on the number of candidates} In the introduction we stated that the number of candidate simplices $\sigma_1$ produced by Webb's algorithm when $d(v,w)=n$ is bounded above by 
\[ 2^{(72g+12)\min\{n-2,21\}} (2^{6g-6}-1). \]
We will now explain this bound; we are grateful to Richard Webb for supplying us with the details.

We can think of the sequence $\tau_{n},\dots,\tau_3$ as a path in the filling multi-arc complex, that is, the simplicial complex whose vertices are simplices of $\A(S')$ whose geometric realizations fill $S'$ and whose edges correspond to simplices with geometric intersection number zero.  Then we obtain $\tau_2$ by extending this path by one more edge and taking some nonempty subset of the simplex of $\A(S')$ represented by the endpoint $\hat \tau_2$ of this extended path.

Webb proved that the degree of an arbitrary vertex of this filling multi-arc complex is bounded above by $2^{72g+12}$ (this is for the case where we start with a closed surface of genus $g$ and cut along a single simple closed curve, as above); see his paper \cite{Webb}.  Our extended path from $\tau_n$ to $\hat \tau_{n-2}$ has length $n-2$ and so this a priori gives a bound of $2^{(72g+12)(n-2)}$ for the number of possibilities for $\hat \tau_2$.  However, there is a version of the bounded geodesic image theorem which tells us that, because the $\tau_i$ arise from a geodesic in $\C(S)$, the actual distance in the filling multi-arc complex between $\tau_n$ and $\hat \tau_2$ is bounded above by 21.  This gives the first multiplicand in the desired bound.  The second multiplicand comes from the number of ways of choosing a nonempty sub-simplex $\tau_2$ of $\hat \tau_2$.  The number of vertices of $\tau_2$ is bounded above by $6g-6$, and so there are $2^{6g-6}-1$ ways to choose $\tau_2$ from $\hat \tau_2$.

\bibliographystyle{plain}
\bibliography{distance}

\end{document}